\documentclass{elsarticle}
\usepackage{mathptmx}
\usepackage{amssymb}
\usepackage{amsfonts}
\usepackage{amsmath}
\usepackage{graphicx}
\usepackage{color}
\usepackage{hyperref}
\newtheorem{thm}{Theorem}[section]
\newtheorem{corollary}[thm]{Corollary}
\newtheorem{lemma}[thm]{Lemma}
\newtheorem{proposition}[thm]{Proposition}
\newdefinition{definition}[thm]{Definition}
\newdefinition{example}[thm]{Example}
\newdefinition{remark}[thm]{Remark}
\newproof{proof}{Proof}
\newproof{pot}{Proof of Theorem \ref{mr1}}
\newcommand{\field}[1]{\mathbb{#1}}

\newcommand{\N }{\field{N}}

\DeclareMathOperator{\Hom}{Hom}
\DeclareMathOperator{\Ker}{Ker}
\DeclareMathOperator{\Max}{Max}
\DeclareMathOperator{\m}{\mathfrak{m}}
\DeclareMathOperator{\Ze}{Z}
\DeclareMathOperator{\Ann}{Ann}

\DeclareMathOperator{\nil}{Nil}

\begin{document}
\begin{frontmatter}
\title{Zaks' conjecture on rings with semi-regular proper homomorphic images}
\author{K. Adarbeh\fnref{fn1}}
\ead{khalidwa@kfupm.edu.sa}
\author{S. Kabbaj\corref{cor}\fnref{fn1}}\cortext[cor]{Corresponding author.}
\ead{kabbaj@kfupm.edu.sa}
\fntext[fn1]{This work was funded by Saudi NSTIP Research Award \# 14-MAT71-04.}
\address{Department of Mathematics and Statistics\\ King Fahd University of Petroleum \& Minerals\\ Dhahran 31261, Saudi Arabia}
\begin{abstract}
In this paper, we prove an extension of Zaks' conjecture on integral domains with semi-regular proper homomorphic images  (with respect to finitely generated ideals) to arbitrary rings (i.e., possibly with zero-divisors). The main result extends and recovers Levy's related result on Noetherian rings \cite[Theorem]{Levy} and Matlis' related result on Pr\"ufer domains \cite[Theorem]{M85}. It also globalizes Couchot's related result on chained rings \cite[Theorem 11]{Cou2003}.  New examples of rings with semi-regular proper homomorphic images stem from the main result via trivial ring extensions.
\end{abstract}


\end{frontmatter}

\section{Introduction}

\noindent Throughout this paper, all rings considered are commutative with identity and all modules are unital. A ring $R$ is \emph{coherent} if every finitely generated ideal of $R$ is finitely presented. The class of coherent rings includes strictly the classes of Noetherian rings, von Neumann regular rings, valuation domains, and semi-hereditary rings. The concept of coherence stemmed up from the study of coherent sheaves in algebraic geometry; and later  developed towards a full-fledged topic in commutative algebra under the influence of homology. For more details on coherence, see please \cite{G1,G2}.

In 1982, Matlis proved that a ring $R$ is coherent if and only if $\hom_{R}(M,N)$ is flat for any injective $R$-modules $M$ and $N$ \cite{M82}. In 1985, he defined a ring $R$ to be \emph{semi-coherent} if $\hom_{R}(M,N)$ is a submodule of a flat $R$-module for any injective $R$-modules $M$ and $N$. Then, inspired by this definition and von Neumann regularity, he defined a ring to be \emph{semi-regular} if any module can be embedded in a flat module. He then provided a connection of this notion with coherence; namely, a ring $R$ is semi-regular if and only if $R$ is coherent and $R_{ M}$ is semi-regular for every maximal ideal $ M$ of $R$. Moreover, he proved that in the class of reduced rings, semi-regularity coincides with von Neumann regularity; and under Noetherian assumption, semi-regularity coincides with self-injectivity \cite{M85}.  It is worth noting, at this point, that the notion of semi-regular ring was briefly mentioned by Sabbagh (1971) in \cite[Section 2]{Sabb} and studied
in non-commutative settings by Jain (1973) in \cite{Jain}, Colby (1975) in \cite{Col}, and Facchini \& Faith (1995) in \cite{FF}, among others, where it was always termed as IF ring. Also, it was extensively studied -under IF terminology- in valuation settings by Couchot in \cite{Cou2003}. Recall here that a semi-regular ring is self fp-injective \cite[Theorem 3.3]{Jain}.

A domain $R$ is \emph{Dedekind} if every ideal of $R$ is projective. In 1966, Levy proved a dual version for this result stating that, for a Noetherian ring $R$ (possibly with zero-divisors),  every proper homomorphic image of $R$ is self-injective if and only if
$R$ is a Dedekind domain or a principal ideal ring with descending chain condition or a local ring whose maximal ideal $ M$ has composition length 2 with $ M^{2}=0$ \cite[Theorem]{Levy}. In 1985, Matlis proved that if $R$ is a Pr\"ufer  domain, then $R/I$ is semi-regular for every nonzero finitely generated ideal $I$ of $R$ \cite[Proposition 5.3]{M85}. Then Abraham Zaks conjectured that the converse of this result should be true; i.e., an integral domain $R$ is  Pr\"ufer if and only if $R/I$ is semi-regular for every nonzero finitely generated ideal $I$ of $R$ . This was proved by Matlis in \cite[Theorem, p. 371]{M85}; extending thus Levy's theorem in the case of integral domains.

In this paper, we prove an extension of Zaks' conjecture on integral domains with semi-regular proper homomorphic images  (with respect to the finitely generated ideals) to arbitrary rings (i.e., possibly with zero-divisors). The main result (Theorem~\ref{mr1}) globalizes Couchot's related result on chained rings \cite[Theorem 11]{Cou2003}; and also extends and recovers Matlis' related result on Pr\"ufer domains (Corollary~\ref{zak}) and Levy's related result on Noetherian rings (Corollary~\ref{levy2}).  New examples of rings with semi-regular proper homomorphic images stem from the main result via trivial ring extensions.

For the reader's convenience, the following diagram of implications summarizes the relations among the main notions involved in this paper:

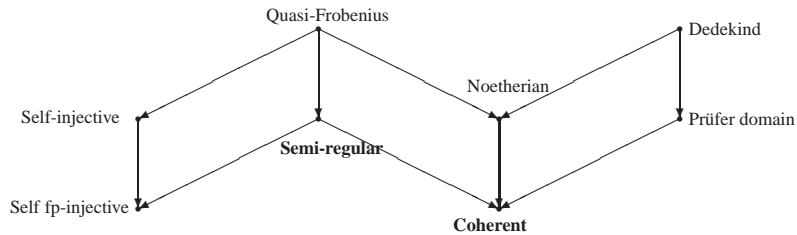
\begin{figure}[h!]
\centering
\[\setlength{\unitlength}{.6mm}
\begin{picture}(120,40)(0,-65)
\put(40,-30){\vector(0,-1){20}}
\put(40,-30){\vector(2,-1){40}}
\put(40,-30){\vector(-2,-1){40}}
\put(0,-50){\vector(0,-1){20}}
\put(80,-50){\vector(0,-1){20}}
\put(120,-30){\vector(-2,-1){40}}
\put(120,-30){\vector(0,-1){20}}
\put(120,-50){\vector(-2,-1){40}}
\put(40,-50){\vector(2,-1){40}}
\put(40,-50){\vector(-2,-1){40}}

\put(40,-30){\circle*{1.2}}\put(42,-29){\makebox(0,0)[b]{\scriptsize Quasi-Frobenius}}
\put(0,-50){\circle*{1.2}} \put(-4,-50){\makebox(0,0)[r]{\scriptsize Self-injective}}
\put(40,-50){\circle*{1.2}}\put(43,-55){\makebox(0,0)[t]{\scriptsize\bf Semi-regular}}
\put(80,-50){\circle*{1.2}} \put(82,-44){\makebox(0,0)[b]{\scriptsize  Noetherian}}
\put(0,-70){\circle*{1.2}} \put(-2,-70){\makebox(0,0)[r]{\scriptsize Self fp-injective}}
\put(80,-70){\circle*{1.2}}\put(78,-72){\makebox(0,0)[t]{\scriptsize\bf  Coherent}}
\put(120,-30){\circle*{1.2}}\put(122,-30){\makebox(0,0)[l]{\scriptsize  Dedekind}}
\put(120,-50){\circle*{1.2}}\put(122,-50){\makebox(0,0)[l]{\scriptsize  Pr\"ufer domain}}
\end{picture}\]
\end{figure}
\bigskip

Throughout,  for a ring $R$, let $Q(R)$ denote its total ring of quotients, $\Ze(R)$ the set of its zero-divisors, and $\Max(R)$ the set of its maximal ideals. For an ideal $I$ of $R$, $\Ann(I)$ will denote the annihilator of $I$.

\section{Main result}\label{mr}

\noindent A ring $R$ is \emph{arithmetical} if every finitely generated ideal of $R$ is locally principal \cite{Fu,Jen,KLS}; and $R$ is a \emph{chained} ring if $R$ is local and arithmetical \cite{BG,BG2,H}. In the domain setting, these two notions coincide with Pr\"ufer and valuation domains, respectively. In \cite{Cou2003}, Couchot investigated  semi-regularity (termed as IF-ring) in the class of chained rings (termed as valuation rings). It is worthwhile recalling that, in a Noetherian setting, semi-regularity coincides with self-injectivity \cite[Proposition 3.4]{M85}; and under coherence, it coincides with the double annihilator condition (i.e., $\Ann(\Ann(I))=\Ann(I)$, for every finitely generated ideal $I$) \cite[Proposition 4.1]{M85}.

Throughout, for a ring $R$ and an $R$-module $M$, $l(M)$ will denote the \emph{composition length} of $M$ ($=\infty$, if $M$ has no composition series).

A ring $R$ is called \emph{residually semi-regular} if $R/I$ is semi-regular, for every nonzero finitely generated proper ideal $I$ of $R$. Levy (resp., Matlis) proved that a Noetherian domain (resp., a domain) $R$ is residually semi-regular if and only if $R$ is Dedekind (resp., Pr\"ufer) \cite[Theorem(1)]{Levy} and \cite[Proposition 5.3]{M85}. In the non-domain setting, Levy's result \cite[Theorem(2)\&(3)]{Levy} also ensures that a Noetherian ring with zero-divisors is residually semi-regular if and only if $R$ is principal Artinian or $(R, M)$ is local with $ M^{2}=0$ and $l( M)=2$. Also, recall Couchot's result that a chained ring is residually semi-regular \cite[Theorem 11]{Cou2003}.

In order to proceed to the main result, we need the notion of residual coherence. Namely, a ring $R$ is \emph{residually coherent} if $R/I$ is coherent, for every nonzero finitely generated proper ideal $I$ of $R$. Obviously, coherent rings and residually semi-regular rings are residually coherent. Also, note that while chained rings are always residually coherent by \cite[Corollary II.14]{Cou2003CA} (or \cite[Theorem 11]{Cou2003}), arithmetical rings are not; see \cite[Theorem II.15]{Cou2003CA} and Example~\ref{rc}.

The following result extends (and solves) Zaks' conjecture to arbitrary rings (i.e., possibly with zero-divisors), generalizing thus Levy's, Matlis', and Couchot's aforementioned results. Recall, for convenience, that a semi-regular ring, being equal to its total ring of quotients, is always a Pr\"ufer ring.

\begin{thm}\label{mr1}
Let $R$ be a ring and consider the following conditions:
\begin{enumerate}
\item[$(\mathcal{C}_{1})$] $(R, M)$ is local with $M^{2}=0$ and $l( M)\leq2$.
\item[$(\mathcal{C}_{2})$] $R$ is arithmetical and residually coherent, and $R_{M}$ is a semi-regular ring for every $M\in\Max(R)$ such that $\Ker(R\rightarrow R_{M})\neq0$.
\end{enumerate}
Then, $R$ is residually semi-regular if and only if $R$ satisfies $(\mathcal{C}_{1})$ or $(\mathcal{C}_{2})$.
\end{thm}

 Notice, at this point, that a coherent arithmetical ring is not residually semi-regular, in general. This is evidenced by Example~\ref{ex1}, which shows that the assumption ``$R_{M}$ is semi-regular for every $M\in\Max(R)$ such that $\Ker(R\rightarrow R_{M})\neq0$" is not redundant with $R$ being arithmetical and residually coherent; and hence a global version for Couchot's result is not always true (even in the class of coherent rings). Moreover, the residual coherence cannot be omitted from $(\mathcal{C}_{2})$ as shown by Example~\ref{rc}, which exhibits an example of an arithmetical and locally semi-regular ring that is not residually coherent (and, a fortiori, not residually semi-regular).

 We break down the proof of the theorem into several lemmas.

\begin{lemma}\label{mr1-lem1}
Let $R$ be a local residually semi-regular ring and let $I_{1}$ and $I_{2}$ be two finitely generated ideals of $R$ with $I_{1}\cap I_{2}\not=0$. Then:
\begin{enumerate}[(i)]
\item $I_{1}\cap I_{2}$ is finitely generated.
\item $I_{1}$ and $I_{2}$ are comparable.
\end{enumerate}
\end{lemma}

\begin{proof}
(i) Let $0\neq x\in I_{1}\cap I_{2}$. Without loss of generality, we may assume that  $Rx\subsetneqq I_{1}\cap I_{2}$ and consider the semi-regular ring $\overline{R}:=R/Rx$ which is coherent by \cite[Proposition 3.3]{M85}. Then, $\overline{I_{1}}\cap \overline{I_{2}}=\overline{I_{1}\cap I_{2}}$ is finitely generated in $\overline{R}$. Hence $I_{1}\cap I_{2}$ is finitely generated in $R$.

(ii) First, note that if $0\not=I, J$, and $K$ are three finitely generated ideals of $R$ with $I\subseteq J$ and $I\subseteq K$, then, by \cite[Proposition 4.1]{M85}, $R/I$ satisfies the double annihilator condition on $J/I$ and $$\Ann_{\frac{R}{I}}(\frac{J}{I})+\Ann_{\frac{R}{I}}(\frac{K}{I})=\Ann_{\frac{R}{I}}(\frac{J\cap K}{I})$$
that is,
\begin{equation}\label{eq1} (I:(I:J))=J\end{equation}
and
\begin{equation}\label{eq2} (I:J)+(I:K)=(I:J\cap K)\end{equation}
where by $(I:J)$ we mean $(I:_{R}J)=\big\{x\in R \mid xJ\subseteq I\big\}$. Now, $0\neq I_{1}\cap I_{2}$ is finitely generated by (i). Hence, by (\ref{eq2}), we obtain
$$ (I_{1}\cap I_{2}:I_{1})+(I_{1}\cap I_{2}:I_{2})=(I_{1}\cap I_{2}:I_{1}\cap I_{2})=R.$$
Therefore, $1=x+y$, for some $x\in (I_{1}\cap I_{2}:I_{1})$ and $y\in (I_{1}\cap I_{2}:I_{2})$. It follows that, for any $a_{1}\in I_{1}$ and $a_{2}\in I_{2}$, we have $(1-y)a_{1}=xa_{1}\in I_{2}\ \text{and}\ ya_{2}\in I_{1}$. Since $R$ is local, either $y$ or $1-y$ is a unit, forcing $I_{1}$ and $I_{2}$ to be comparable.
\hfill$\Box$\end{proof}

\begin{lemma}\label{mr1-lem2}
Let $(R,M)$ be a local residually semi-regular ring and let $x,y\in R$.
\begin{enumerate}[(i)]
\item $x^{2}\not=0$ and $y^{2}\not=0$ $\Rightarrow$ $xy\not=0$ $\Rightarrow$ $(x)$ and $(y)$ are comparable.
\item $x^{2}=0$ and $y^{2}\not=0$ $\Rightarrow$ $(x)\subseteq (y)$.
\end{enumerate}
\end{lemma}

\begin{proof}
 (i) In view of Lemma~\ref{mr1-lem1}, we only need to prove the first implication. Assume $x^{2}\not=0$ and $y^{2}\not=0$. Clearly, $x\not=0$ and $y\not=0$. Suppose, by way of contradiction, that $xy=0$. Then, necessarily, $(x)$ and $(y)$ are incomparable. Next, let $I:=(x,y)$. Then, for any $z\in\Ann(y)$, $(x,z)\cap I\neq0$. By Lemma~\ref{mr1-lem1}, $z\in I$ since $y\notin (x,z)$. Therefore, $\Ann(y)\subseteq I$. Further, $y\notin I^{2}$; otherwise, $y=ax^{2}+by^{2}$ for some $a,b\in R$ yields $y=ax^{2}(1-by)^{-1}\in (x)$, absurd. So, $\overline{y}\not=0$ in $\overline{R}:=R/I^{2}$. We claim that $\Ann(\overline{y})=\overline{I}$ in $\overline{R}$. Indeed, let $\overline{t}\in\Ann(\overline{y})$. Then, there exist  $a,b\in R$ such that $y(t-by)=ax^{2}\in (x)\cap(y)$. By Lemma~\ref{mr1-lem1}, $y(t-by)=0$. Hence $t-by\in\Ann(y)\subseteq I$. Hence $\overline{t}\in \overline{I}$. The reverse inclusion is obvious, proving the claim. Now, the fact that $\overline{R}$ is semi-regular yields
\begin{equation}\label{eq3} \overline{I}\subseteq\Ann(\overline{I})=\Ann(\Ann(\overline{y}))=(\overline{y})\subseteq\overline{I}.\end{equation}
It follows that $(\overline{y})=\overline{I}$ and therefore
\begin{equation}\label{eq4} I=(y)+I^{2}=(y)+MI.\end{equation}
By Nakayama's lemma, we get $I=(y)$, the desired contradiction.

(ii) Assume $x^{2}=0$ and $y^{2}\not=0$. Clearly, $y\not=0$. Without loss of generality, we may assume  $x\not=0$ and $y$ is not a unit. If $xy\not=0$, then $(x)$ and $(y)$ are comparable and necessarily $(x)\subseteq (y)$. Next, suppose that $xy=0$ and let $I:=(x,y)$. Similarly to (i), we have $\Ann(y)\subseteq I$, and  $\overline{y}\not=0$ in $\overline{R}:=R/I^{2}$; otherwise, $y=ay^{2}$ for some $a\in R$ yields $y(1-ay)=0$, absurd (since $1-ay$ is a unit). Also, $ty=ay^{2}$ for some $a\in R$ yields $t-ay\in\Ann(y)\subseteq I$ and so $t\in I$. That is, $\Ann(\overline{y})=\overline{I}$ in $\overline{R}$. Similar arguments as in (\ref{eq3}) and (\ref{eq4}) lead to  $I=(y)$, as desired.
\hfill$\Box$\end{proof}

\begin{lemma}\label{mr1-lem3}
Let $R$ be a local residually semi-regular ring and $I$ a finitely generated ideal of $R$. Then, either $I$ is principal or $I$ is generated by two elements with $I^{2}=0$.
\end{lemma}

\begin{proof}
Notice first that, for any $0\not=x,y,z\in R$, $(x,y)$ and $(x,z)$ are comparable by Lemma~\ref{mr1-lem1}. It follows that any finitely generated ideal is generated by at most two elements. So, $I=(x,y)$ for some $x,y\in R$. If $xy\not=0$ or $x^{2}\not=0$ or $y^{2}\not=0$, then $I$ is principal by Lemma~\ref{mr1-lem1} and Lemma~\ref{mr1-lem2}, completing the proof of the lemma.
\hfill$\Box$\end{proof}

Recall that a ring $R$ is Gaussian if $c(fg)=c(f)c(g)$ for any polynomials $f,g$ in $R[X]$, where $c(f)$ denotes the content of $f$ (i.e., the ideal of $R$ generated by the coefficients of $f$). The class of Gaussian rings lies strictly between the two classes of arithmetical rings and Pr\"ufer rings \cite{BG,BG2,T}.

\begin{lemma}\label{mr1-lem4}
Let $(R,M)$ be a local residually semi-regular ring. Then, $R$ is Gaussian. Moreover, if $\Ze(R)$ is not uniserial and $(\Ze(R))^{2}=0$, then $\Ze(R)=M$.
\end{lemma}

\begin{proof}
By \cite[Theorem 2.2]{BG2}, $R$ is Gaussian if and only if, $\forall\ a,b\in R$, $(a,b)^{2}= (a^{2})$ or $(b^{2})$, and if $(a,b)^{2}= (a^{2})$ and $ab=0$, then $b^{2}=0$. Next, let $a,b\in R$. The case $a^{2}\not=0$ and $b^{2}\not=0$ is handled by Lemma~\ref{mr1-lem2}(i) and the case $a^{2}\not=0$ and $b^{2}=0$ is handled by Lemma~\ref{mr1-lem2}(ii). If $a^{2}=b^{2}=0$, then $ab=0$ by Lemma~\ref{mr1-lem1}, whence  $(a,b)^{2}=0$, completing the proof of the first statement.

Next, suppose that $\Ze(R)$ is not uniserial and $(\Ze(R))^{2}=0$. The latter assumption yields $\Ze(R)=\Ann(a)$, for every $a\in\Ze(R)$. Further, by Lemma~\ref{mr1-lem1}, there exist two nonzero elements $a,b\in\Ze(R)$ with $(a)\cap(b)=0$. So, we obtain
$$\Ze(R)=\Ann(b)=((a)\cap(b):b)=((a):b).$$
Since $R/(a)$ is semi-regular and, hence, coherent, we deduce that $\Ze(R)$ is finitely generated. Further, $\Ze(R)$ is a prime ideal since $R$ is local Gaussian. It follows that $R/\Ze(R)$ is a semi-regular integral domain and, hence, a field. That is, $\Ze(R)=M$, completing the proof of the lemma.
\hfill$\Box$\end{proof}

Let $R$ be a ring and $M$ an $R$-module. An $R$-module $V$ is \emph{$M$-projective} if the natural map $\Hom_{R}(V,M)\rightarrow\Hom_{R}(V,M/N)$ is surjective for every submodule $N$ of $M$; and $V$ is \emph{quasi-projective} if $V$ is $V$-projective. A ring $R$ is an \emph{fqp-ring} if every finitely generated ideal of $R$ is quasi-projective \cite{AJK,Cou2015}. We always have:
\begin{center}Arithmetical $\Rightarrow$ fqp $\Rightarrow$ Gaussian\end{center} and the fqp notion is a local property in the class of coherent rings \cite[Proposition 4.4]{Cou2015} or \cite[Corollary 3.15]{AJK}.

\begin{lemma}\label{mr1-lem7}
Let $(R,M)$ be a local residually semi-regular ring. Then, $\Ann(x)=\Ann(y)$, for any nonzero $x,y\in R$  such that $(x)$ and $(y)$ are incomparable.
\end{lemma}

\begin{proof}
If $M^{2}=0$, then $M=\Ann(x)$ for every $x\in M$ and the result trivially holds. Next, assume $M^{2}\not=0$ and let $x,y$ be two nonzero elements of $R$  such that $(x)$ and $(y)$ are incomparable. By Lemmas \ref{mr1-lem1} and \ref{mr1-lem2}, we get
$$(x)\cap(y)=0\ \text{and}\ x^{2}=y^{2}=xy=0.$$
Hence $x,y\in\Ann(x)\cap\Ann(y)$. Further, as seen in the proof of Lemma~\ref{mr1-lem4}, $\Ann(x)=((x)\cap(y):x)=((y):x)$ is finitely generated; and likewise so is $\Ann(y)$. Hence, by Lemma~\ref{mr1-lem1}, $\Ann(x)$ and $\Ann(y)$ are comparable; say, $\Ann(x)\subseteq\Ann(y)$. Next, we prove the reverse inclusion. Let $t\in\Ann(y)$ and assume, by way of contradiction, that $tx\not=0$. First, notice that, via (\ref{eq1}), we have
$$(tx,y)\subseteq\Big((tx):\big((tx):(y)\big)\Big)\subseteq\Big((tx):\big((tx):(tx,y)\big)\Big)=(tx,y).$$
Moreover, $\big((tx):(x)\big)$ and $\big((tx):(y)\big)$ are finitely generated by coherence of $R/(tx)$; and $0\neq x\in\big((tx):(x)\big)\cap\big((tx):(y)\big)$. So, by Lemma~\ref{mr1-lem1}, $\big((tx):(x)\big)$ and $\big((tx):(y)\big)$ are comparable. If $\big((tx):(x)\big)\subseteq\big((tx):(y)\big)$, then we obtain via (\ref{eq1})
$$(tx,y)=\Big((tx):\big((tx):(y)\big)\Big)\subseteq\Big((tx):\big((tx):(x)\big)\Big)=(x)$$
yielding $(y)\subseteq(x)$, absurd. So, suppose $\big((tx):(y)\big)\subseteq\big((tx):(x)\big)$. Then same argument as above yields $(x)\subseteq (tx,y)$. That is, $x-atx\in(x)\cap(y)=0$, for some $a\in R$. Hence, $x(1-at)=0$, whence $1-at\in\Ann(x)\subseteq\Ann(y)$. It follows that $y=yat=0$, absurd.
\hfill$\Box$\end{proof}

\begin{lemma}\label{mr1-lem6}
A local residually semi-regular ring is an fqp-ring.
\end{lemma}

\begin{proof}
Let $I$ be a finitely generated ideal of $R$. We shall prove that $I$ is quasi projective. By \cite[Theorem 2.3]{AJK}, we only need to prove that $I\cong \left(\frac{R}{J}\right)^{n}$, for some ideal $J$ of $R$ and integer $n\geq 0$. By Lemma~\ref{mr1-lem3}, either $I$ is principal or $I$ is generated by two elements with $I^{2}=0$. If $I=Rx$, then $I\cong \frac{R}{\Ann(x)}$, as desired. Next, suppose that $I=(x,y)$ is not principal. We claim that
$$I\cong \left(\frac{R}{\Ann(x)}\right)^{2}.$$
To this purpose, consider the surjective $R$-map $\varphi: R^{2}\rightarrow I$ defined by $\varphi(a,b)=ax+by$. Now, $\varphi(a,b)=0$ yields $ax=-by\in(x)\cap(y)=0$ by Lemma~\ref{mr1-lem1} since  $(x)$ and $(y)$ are incomparable. Therefore, $a\in\Ann(x)$ and $b\in\Ann(y)=\Ann(x)$ by Lemma~\ref{mr1-lem7}. It follows that
$$\Ker(\varphi)=\Ann(x)\times\Ann(x)$$
and thus
$$I\cong \frac{R^{2}}{\Ann(x)\times\Ann(x)}\cong \left(\frac{R}{\Ann(x)}\right)^{2}$$ completing the proof of the lemma.
\hfill$\Box$\end{proof}

\begin{lemma}\label{levy1}
Let $(R,M)$ be a local ring with $M^{2}=0$. Then, $R$ is semi-regular if and only if $l(M)\leq 1$.
\end{lemma}

\begin{proof}
Assume $R$ is semi-regular. We may assume that $R$ is not a field and let $0\neq x\in M$. Then, we have
$$
\begin{array}{ccl}
xR 	            &=	&\Ann_{R}(\Ann_{R}(xR))\\
                &=	&\Ann_{R}(M)\\
                &=  &M.
\end{array}
$$
Consequently, $l(M)=1$. Conversely, assume that $l(M)= 1$ and let $0\neq x\in M$. Then, we have
$$
\begin{array}{ccl}
xR 	    &=  &M\\
        &=	&\Ann_{R}(M)\\
        &=	&\Ann_{R}(\Ann_{R}(xR)).
\end{array}
$$
It follows that $R$ satisfies the double annihilator condition on finitely generated ideals. Since $R$ is coherent (in fact, principal), then $R$ is semi-regular by \cite[Proposition 4.1]{M85}.
\hfill$\Box$\end{proof}

Finally, we proceed to the proof of the theorem.

\begin{pot}
We first prove sufficiency. Let $I$ be a nonzero finitely generated proper ideal of $R$.
Assume that $(\mathcal{C}_{1})$ holds. Therefore, $l(M/I)\leq 1$ and hence, by Lemma~\ref{levy1}, $R/I$ is a semi-regular ring, as desired. Next, assume that $(\mathcal{C}_{2})$ holds. Then, $R/I$ is coherent. Next, let $M\in\Max(R)$ with $I\subseteq M$ and $IR_{M}=rR_{M}$, for some $0\neq r\in R$. If $rR_{M}\not=0$, then
$$(R/I)_{M/I}\cong R_{M}/rR_{M}$$
is semi-regular by \cite[Theorem 11(1)]{Cou2003}. If $rR_{M}=0$, then
$$(R/I)_{M/I}\cong R_{M}$$
is semi-regular by hypothesis. Therefore, by \cite[Proposition 2.3]{M85}, $R$ is a residually semi-regular ring.

Conversely, assume $R$ is  residually semi-regular and let us envisage two cases.

{\bf Case 1:} Assume there is $M\in\Max(R)$ such that $M^{2}=0$. Necessarily, $(R,M)$ is local with $M$ being the only prime ideal of $R$. We will show that either $R$ is a chained ring or $l(M)=2$. Without loss of generality, we may assume that $R$ is not a field (i.e., $M\not=0$). If $(a)\cap (b) \not= 0$ for every nonzero $a,b\in M$, then, by Lemma~\ref{mr1-lem1}, $R$ is a chained ring. Further, let $I$ be a nonzero proper ideal of $R$, $0\not=a\in I$, and $x\in M$. Then either $x\in (a)$ or $a\in (x)$. The second case yields $a=ux$ for some unit $u\in R$, hence $I=M=(a)$; i.e., $l(M)=1$. By Lemma~\ref{levy1}, $R$ is semi-regular so that $(\mathcal{C}_{2})$ is satisfied. Next, assume that there exist nonzero $a_{o},b_{o}\in M$ such that
$$(a_{o})\cap (b_{o})=0.$$
Then, $(a_{o})$ and $(b_{o})$ are incomparable and, moreover, the assumption $ M^{2}=0$ yields the following property for any $0\not=a,b\in M$:
\begin{equation}\label{eq5} (b)\nsubseteq(a)  \Rightarrow  M=(a,b).\end{equation}
Indeed, we obviously have
$$M\subseteq (a:b)\ \text{and}\ M\subseteq (a:M).$$
Hence $(a:b)=M$ since $(a:b)\not=R$ and whence $(a:(a:b))=(a:M)=M$ since $(a: M)\not=R$. So, we obtain
$$M = (a:(a:b))\subseteq  (a:(a:(a,b))) = (a,b)\subseteq   M$$
where the second equality is ensured by (\ref{eq1}), yielding $ M=(a,b)$, as claimed. It follows that $ M=(a_{o},b_{o})$ and thus $R$ is Artinian. Hence
$$2\leq l( M)<\infty.$$
Next, let $I$ be an ideal of $R$ with $0\subsetneqq I\subsetneqq  M$ and let $0\not=a\in I$. Therefore, for any $b\in R$, if $b\notin (a)$, then $ M=(a,b)$ by (\ref{eq5}). It follows that $I=(a)$ and no ideal can be inserted between $I$ and $ M$. Consequently, $l( M)=2$ so that $(\mathcal{C}_{1})$ is satisfied.

{\bf Case 2:} Assume that $M^{2}\not=0$, for every $M\in\Max(R)$ (and observe that $M^{2}R_{M}$ might be null). Let $M\in\Max(R)$ and, without loss of generality, assume that $R_{ M}$ is not a field. Note first that if $r/1=0$ for some nonzero  $r\in R$, then
$$R_{M}\cong R_{M}/rR_{M}\cong(R/rR)_{M/rR}$$
is semi-regular, as desired. It remains to show that $R_{M}$ is a chained ring. To this purpose, let us envisage two subcases. {\sc Subcase 2.1:} Suppose that $M^{2}R_{M}=0$. Necessarily, $R_{M}\cong (R/M^{2})_{M/M^{2}}$ is semi-regular. Hence, by Lemma~\ref{levy1}, $l(MR_{M})=1$; whence $xR_{M}=MR_{M}$, for any $0\neq x\in MR_{M}$. In particular, $R_{M}$ is a chained ring. {\sc Subcase 2.2:} Suppose that $M^{2}R_{M}\neq0$. By Lemma~\ref{mr1-lem6}, $R_{M}$ is an fqp-ring. Assume, by way of contradiction, that $R_{M}$ is not a chained ring. Then, by \cite[Lemmas 3.12 \& 4.5]{AJK}, we have
$$(\nil(R_{M}))^{2}=0\ \text{and}\ \Ze(R_{M})=\nil(R_{M}).$$
That is, $(\Ze(R_{M}))^{2}=0$. But, by Lemma~\ref{mr1-lem2}, $\Ze(R_{M})$ is not uniserial and therefore, by Lemma~\ref{mr1-lem4}, $\Ze(R_{M})=MR_{M}$, the desired contradiction. So, in both cases $R_{M}$ is a chained ring. Therefore, $R$ is an arithmetical ring and, consequently, $(\mathcal{C}_{2})$ is satisfied (since $R$ is trivially residually coherent).
\hfill$\Box$\end{pot}

As a first application of Theorem~\ref{mr1}, the next corollary handles the special case of reduced rings.

\begin{corollary}\label{reduced}
Let $R$ be a reduced ring. Then, $R$ is residually semi-regular if and only if $R$ is either a Pr\"ufer domain or a von Neumann regular ring.
\end{corollary}

\begin{proof}
An arithmetical reduced ring has weak global dimension $\leq 1$ \cite[Theorem 3.5]{BG}, and hence it is locally a (valuation) domain \cite[Theorem 3.4]{BG}.
A combination of this result with the basic fact ``that a semi-regular domain is a field" leads to the conclusion via Theorem~\ref{mr1}.
\end{proof}

As a straightforward application of Theorem~\ref{mr1} or Corollary~\ref{reduced}, we recover Matlis' result which solved Zaks' conjecture on residually semi-regular domains.

\begin{corollary}[{\cite[Theorem, p. 371]{M85}}]\label{zak}
A  domain $R$ is residually semi-regular if and only if $R$ is Pr\"ufer.
\end{corollary}

Next, we recover Levy's result on Noetherian rings with self-injective proper homomorphic images. In this vein, recall for convenience that, under Noetherian assumption, semi-regularity coincides with self-injectivity.

\begin{corollary}[{\cite[Theorem]{Levy}}]\label{levy2}
Let $R$ be a Noetherian ring and consider the following conditions:
\begin{enumerate}
\item[$(\mathcal{C}_{1})$] $R$ is a Dedekind domain.
\item[$(\mathcal{C}_{2})$] $R$ is a principal Artinian ring.
\item[$(\mathcal{C}_{3})$] $(R, M)$ is local with $M^{2}=0$ and $l( M)=2$.
\end{enumerate}
Then, $R$ is residually semi-regular if and only if $R$ satisfies $(\mathcal{C}_{1})$ or $(\mathcal{C}_{2})$ or $(\mathcal{C}_{3})$.
\end{corollary}

\begin{proof}
In view of Corollary~\ref{zak}, we may assume that $R$ is not a domain. For sufficiency, it suffices to consider the case where  $R$ is principal Artinian. Then, obviously, $R$ is arithmetical. Moreover, let $M\in\Max(R)$. Then, $MR_{M}=(t)$ for some $0\neq t\in R_{M}$ with $t^{n}=0$ for some minimal integer $n\geq2$. So, the only nonzero ideals of $R_{M}$ are $(t^{k})$ where $k=1,\dots,n-1$, and one can easily check that $$\Ann_{R_{M}}(\Ann_{R_{M}}(t^{k}))=\Ann_{R_{M}}(t^{n-k})=(t^{k}).$$
Therefore, $R_{M}$ is semi-regular and thus Theorem~\ref{mr1} leads to the conclusion. For necessity, in view of Theorem~\ref{mr1}, we only need to consider the case when $R$ is an arithmetical residually semi-regular ring and check that $R$ is principal Artinian. Indeed, let $M\in\Max(R)$. So, $R_{M}$ is a chained Noetherian ring. If $R_{M}$ is a domain, then it is semi-regular (since $R$ is not a domain) and a fortiori a field. If $R_{M}$ is not a domain, assume $P$ is a non-maximal prime ideal of $R_{M}$. Then,
$$0\subsetneqq P\subseteq \cap_{n\geq1} M^{n}R_{M}=0$$
which is absurd. So, in both cases, we have $\dim(R_{M})=0$. Consequently, $\dim(R)=0$
and thus $R$ is Artinian. It follows that $R$ is principal by the structure theorem for Artinian rings (since the arithmetical property is stable under factor rings), completing the proof of the corollary.
\hfill$\Box$\end{proof}

Another application of Theorem~\ref{mr1}  shows that, in the class of semi-regular rings, the arithmetical property coincides with the notion of residually semi-regular ring.

\begin{corollary}\label{sr}
Let $R$ be a semi-regular ring. Then, $R$ is arithmetical if and only if $R$ is residually semi-regular.
\end{corollary}

\begin{proof}
 Combine Theorem~\ref{mr1} with Lemma~\ref{levy1} for sufficiency and \cite[Proposition 2.1]{M85} for necessity.
\hfill$\Box$\end{proof}

We will appeal to this corollary, in the next section, to provide new examples of residually semi-regular rings, arising as arithmetical semi-regular rings.

\section{Examples}\label{ex}

\noindent We first provide an example of a coherent arithmetical ring which is not residually semi-regular. This shows that the assumption ``$R_{M}$ is semi-regular for every $M\in\Max(R)$ such that $\Ker(R\rightarrow R_{M})\neq0$" within Condition $(\mathcal{C}_{2})$ of Theorem~\ref{mr1} is not redundant with the arithmetical property; and then Couchot's result \cite[Theorem 11]{Cou2003} that ``\emph{a chained ring is residually semi-regular}" does not carry up to (coherent) arithmetical rings.

 A ring is \emph{semi-hereditary} if all its finitely generated ideals are projective. We have the following (irreversible) implications \cite{BG,G1,G2}:
\bigskip

$$\begin{array}{rcccl}
\text{\small Pr\"ufer domain}           &           &                                       &               &\text{\small arithmetical ring}\\
                                        &\searrow   &                                       &\nearrow       &\\
                                        &           &\text{\small semi-hereditary ring}     &\rightarrow    &\text{\small reduced ring}\\
                                        &\nearrow   &                                       &\searrow       &\\
\text{\small von Neumann regular ring}  &           &                                       &               &\text{\small coherent ring}
\end{array}$$
\bigskip

\begin{example}\label{ex1}
A straightforward application of Corollary~\ref{reduced} shows that any semi-hereditary ring with zero-divisors which is not a von Neumann regular ring is a basic example of a coherent arithmetical ring that is not residually semi-regular.
\end{example}

The next example shows that the residual coherence cannot be omitted from Condition $(\mathcal{C}_{2})$ of Theorem~\ref{mr1}; namely, we exhibit an arithmetical and locally semi-regular ring that is not residually coherent (and, a fortiori, not residually semi-regular).

\begin{example}\label{rc}
We borrow our construction from \cite[Example 2.5]{AK}. Let $k$ be a field, $A:=\prod_{i\in\N}F_{i}$ and $I:=\bigoplus_{i\in\N}F_{i}$, where $F_{i}=k\ \forall\ i\in\N$. Let $R:=A\ltimes~\frac{A}{I}$ be the trivial ring extension of $A$ by $A/I$. Any prime ideal $P$ of $R$ has the form $P:=p\ltimes~\frac{A}{I}$, for some prime ideal $p$ of $A$. So, we have $R_{P}\cong A_{p}\ltimes~\frac{A_{p}}{I_{p}}$, which is isomorphic to $k$ if $I\nsubseteq p$ or to $k\ltimes~k$ if $I\subseteq p$. But, $k\ltimes~k$ is clearly principal (since it has a unique nonzero proper ideal $0\ltimes~k$) and quasi-Frobenius by \cite[Theorem 3.6]{Kou} or \cite[Corollary 2.2]{AK}. It follows that $R$ is arithmetical and locally semi-regular, as desired. Now, let $0\not=x\in I$ and observe that $(x,\overline{0})R=xA\ltimes (\overline{0})$ so that $\Ann_{\frac{R}{(x,\overline{0})R}}\left(\overline{(0,\overline{1})}\right)=\frac{I\ltimes~\frac{A}{I}}{(x,\overline{0})R}$ is not finitely generated in $\frac{R}{(x,\overline{0})R}$ since $I\ltimes~\frac{A}{I}$ is not finitely generated in $R$ (due to the fact that $I$ is not finitely generated in $A$). So, $R$ is not residually coherent, as desired.
\end{example}

Next, we use Theorem~\ref{mr1} to construct original examples of non-local coherent residually semi-regular rings beyond Matlis', Levy's, and Couchot's contexts. For this purpose, we investigate the transfer of this notion to trivial extensions.

Recall that the trivial extension of a ring $A$ by an $A$-module $E$ is the ring $R:= A\ltimes E$, where the underlying group is $A \times E$ and the multiplication is given by $(a,e)(b,f) = (ab, af+be)$. The ring $R$ is also called the (Nagata) idealization of $E$ over $A$ and is denoted by  $A(+)E$ \cite{AW,Na}. For more details on trivial ring extensions, we refer the reader to Glaz's book \cite{G1} and Huckaba's books \cite{H}. Recent works investigating various ring-theoretic aspects of these constructions are  \cite{AW,BDS,KLS,Kou,MaMi,NY,Olb}.

Let us first recall an important result from \cite{AK} which establishes the transfer of coherence to trivial ring extensions issued from domains. In this result, we use Fuchs-Salce's definition of a \emph{coherent module}; that is, all its finitely generated submodules are finitely presented \cite[Chapter IV]{FuSa} (i.e., the module itself doesn't have to be finitely generated).

\begin{lemma}[{\cite[Proposition 3.5]{AK}}]\label{coherence}
Let $A$ be a domain which is not a field, $E$ a divisible $A$-module, and $R:=A\ltimes~E$. Then, $R$ is coherent if and only if  $A$ is coherent, $E$ is torsion coherent, and $\Ann_{E}(x)$ is finitely generated for all $x\in A$.
\end{lemma}

The next result investigates the transfer of the notion of residually semi-regular ring to trivial ring extensions issued from local rings.

\begin{proposition}\label{trivial}
Let $(A,M)$ be a local ring, $E$ a nonzero $A$-module, and $R:=A\ltimes~E$. Consider the following conditions:
\begin{enumerate}
\item[$(\mathcal{C}_{1})$] $A$ is a field and $\dim_{A}(E)\leq 2$.
\item[$(\mathcal{C}_{2})$] $M^{2}=0$ with $l( M)=1$ and $E\cong A/M$.
\item[$(\mathcal{C}_{3})$] $A$ is a non-trivial valuation domain, $E$ is a uniserial divisible torsion coherent module, and $\Ann_{E}(x)$ is finitely generated for all $x\in A$.
\end{enumerate}
Then, $R$ is a coherent residually semi-regular ring if and only if any one of the above three conditions holds.
\end{proposition}

\begin{proof}
 Assume that $R$ is residually semi-regular. By Theorem~\ref{mr1}, $(M\ltimes E)^{2}=0$ with $l(M\ltimes E)=2$ or $R$ is a chained ring. The first case yields $M^{2}=0$ and $ME=0$ (i.e., $E$ is an $A/M$-vector space) with $l(M)+l(E)=2$. It follows that either $A$ is a field with $l(E)=2$ (i.e., $\dim_{A}(E)=2$) or $l(M)=1$ and $l(E)=1$ (i.e., $E\cong A/M$). Next, assume that $R$ is a chained ring. If $A$ is a field, then $\dim_{A}(E)=1$ by \cite[Theorem 3.1]{BKM}. If $A$ is not a field, then a combination of \cite[Proposition 1.1]{Cou2015} and Lemma~\ref{coherence} leads to the conclusion. Conversely, suppose that $(\mathcal{C}_{1})$ or $(\mathcal{C}_{2})$ holds. Then, $R$ is coherent by  \cite[Theorem 2.6]{KM2}, and $(M\ltimes E)^{2}=0$ with $l(M\ltimes E)=2$. By Theorem~\ref{mr1}, $R$ is residually semi-regular. Next, suppose that $(\mathcal{C}_{3})$ holds. By Lemma~\ref{coherence}, $R$ is coherent and, by \cite[Proposition 1.1]{Cou2015}, $R$ is  a chained ring and hence residually semi-regular by Theorem~\ref{mr1}.
\hfill$\Box$\end{proof}

Notice that coherent residually semi-regular rings issued via $(\mathcal{C}_{1})$ or $(\mathcal{C}_{2})$ of Proposition~\ref{trivial} are necessarily Noetherian. However, one may use $(\mathcal{C}_{3})$ to provide examples of non-local non-Noetherian coherent residually semi-regular rings with zero-divisors (i.e., beyond Matlis', Levy's, and Couchot's contexts), as shown below.

\begin{example}\label{ex2}
Let $A$ be a non-local non-Noetherian Pr\"ufer domain, $E:=\frac{Q(A)}{A}$, and $R:=A\ltimes E$. Then $R$ is a non-local non-reduced non-Noetherian coherent residually semi-regular ring. Indeed, $R$ is not reduced (as it is the case of any trivial extension) and it is neither local nor Noetherian since $A$ is not. Moreover, $R$ is a semi-regular (and, a fortiori, coherent) ring by \cite[Example 3.12]{AK}. Next, let $M\in\Max(R)$. Then, $M=\m\ltimes E$, for some maximal ideal $\m$ of $A$ and hence
$$R_{M}=A_{\m}\ltimes E_{\m}=A_{\m}\ltimes\frac{Q(A_{\m})}{A_{\m}}$$
with $A_{\m}$ being a valuation domain. Now, $Q(A_{\m})$ is a coherent $A_{\m}$-module (since it is torsion-free) and so is $E_{\m}$. Moreover, $E_{\m}$ is clearly a divisible torsion module and
$$\Ann_{E_{\m}}(x)=\overline{(1/x)}A_{\m}$$
for any nonzero $x\in A_{\m}$. It follows that $R_{M}$ is residually semi-regular by Proposition~\ref{trivial}. Consequently, $R$ is locally residually semi-regular and hence residually semi-regular by Corollary~\ref{sr}, since semi-regularity is stable under localization and the arithmetical notion is a local property.
\end{example}

\section*{Acknowledgments}

\noindent We are grateful to the referee for crucial suggestions and comments which enabled to remove the ``coherence" assumption in the initial statements of Lemmas 2.3, 2.5, and 2.6, and hence of the main theorem; and then replace it with the residual coherence in Condition $(\mathcal{C}_{2})$. Particularly, we owe to the referee the proof of the second statement of Lemma 2.5.


\end{document}